\documentclass[11pt,oneside,reqno]{amsart}

\advance\topmargin by -0.5in 
\advance\textheight by +1.1in
\advance\oddsidemargin by -0.6in \advance\textwidth by +1.2in

\usepackage{amssymb,amsmath,amscd,amsthm}
\usepackage{hyperref}
\usepackage{empheq}
\usepackage{cases}
\usepackage{color}
\usepackage[all]{xy}
\usepackage{enumerate}
\usepackage{enumitem}

\newcommand{\df}{\displaystyle\frac}
\newcommand{\ds}{\displaystyle\sum}

\newcommand{\C}{{\mathbb C}}
\newcommand{\CP}{{\mathbb {CP}}}

\newcommand{\HP}{{\mathbb {HP}}}
\newcommand{\OP}{{\mathbb {OP}}}

\newcommand{\Q}{{\mathbb Q}}

\newcommand{\Z}{{\mathbb Z}}

\newcommand{\LP}{{\mathcal{L}}}
\newcommand{\AP}{{\hat{A}}}

\newcommand{\td}{\textup{Td}}

\makeatletter
\newtheorem*{rep@theorem}{\rep@title}
\newcommand{\newreptheorem}[2]{%
\newenvironment{rep#1}[1]{%
 \def\rep@title{#2 \ref{##1}}%
 \begin{rep@theorem}}%
 {\end{rep@theorem}}}
\makeatother

\newtheorem{theorem}{Theorem}

\newtheorem{maintheorem}{Theorem}

\newreptheorem{maintheorem}{Theorem}

\newreptheorem{maincorollary}{Corollary}

\makeatletter
\@addtoreset{theorem}{chapter}
\makeatother
\newtheorem{proposition}[theorem]{Proposition}

\newtheorem{corollary}[theorem]{Corollary}

\theoremstyle{definition}

\newtheorem{remark}[theorem]{Remark}

\title{Almost complex manifold with Betti number $b_i=0$ except $i=0, n/2, n$ }

\author{}
\author{Zhixu Su}
\address{Department of Mathematics, University of Washington Seattle}\email{zhixusu@uw.edu}\urladdr{}

\begin{document}

\begin{abstract} This paper studies existence of $n=4k\  (k>1)$ dimensional simply-connected closed almost complex manifold with Betti number $ b_i=0$ except $i=0, n/2, n$. We characterize all the rational cohomology rings of such manifolds and show they must have even Euler characteristic and even signature, which is to say the middle Betti number $b_{n/2}$ must be even. Parallel to the author's earlier work on realizing rational cohomology ring by smooth closed manifolds, we state and prove Sullivan's rational surgery realization theorem for almost complex manifold and demonstrate its application in our context. A prescribed rational cohomology ring can be realized by a simply connected almost complex manifold if and only if the ring structure supports the intersection form of a closed manifold, and it holds Chern numbers that satisfy the signature equation and the Riemann--Roch integrality relations, and the top Chern number equals the Euler characteristic.  According to Stong's characterization of $U$ and $SU$ cobordism, we explicitly compute the Riemann--Roch integrality relations among Chern numbers in the case when only the middle and top Chern classes can be nonzero. The necessary and sufficient conditions for realization are expressed as a set of congruence relations among the signature and Euler characteristic, we show that the lower bounds of the 2--adic order of the signature and the Euler characteristic increase with respect to the dimension of the realizing manifold.

\end{abstract}

\maketitle

\vspace{-0.8cm}
\section{introduction}

Let $H^*=\bigoplus_{i=0}^n H^i$ be a $1$--connected $\Q$ Poincar\'{e} duality algebra of dimension $n=4k \,(k>1)$, i.e., $H^*=\bigoplus_{i=0}^n H^i$ is a graded commutative algebra over $\Q$ with $H^i$ finite dimensional for each $i$, $H^1=0$, $H^0\cong H^n=\Q$, and the multiplication defines a non--degenerate bilinear mapping $H^i \times H^{n-i}\to H^n(\cong \Q)$ sending $(\alpha, \beta)\mapsto \alpha\beta$, which induces the Poincar\'{e} duality isomorphism $H^{n-i}\cong (H^i)^*$. 


Prescribing a $1$--connected $\Q$ Poincar\'{e} duality algebra $H^*$, does there exist a simply connected smooth closed manifold $M^n$ whose rational cohomology ring $H^*(M;\Q)\cong H^*$? Sullivan's rational surgery realization theorem (\cite[Theorem 13.2]{Sullivan77}, \cite[Theorem 1]{Barge76}) provided necessary and sufficient conditions to answer this question. There has been work studying realization of certain prescribed algebras. In \cite{PL07}, Papadima and P\v{a}unescu analyzed the smoothing problem of Artinian complete intersection below dimension $12$. In the author's PhD thesis \cite{Su09} under the supervision of Jim Davis, and later paper \cite{Su14}, Sullivan's theorem was applied to study existence dimensions of rational projective plane, which is named for smooth closed manifolds with rational cohomology ring $H^*(M^n;\Q)\cong \Q[\alpha]/\langle \alpha^3\rangle, |\alpha|=n/2$, such manifolds have Betti numbers $b_i=0$ except $b_0=b_{n/2}=b_n=1$. The author's thesis gave a concrete proof illustrating how a rational degree one normal map was obtained in rational surgery if the rational cohomology ring supports Pontryagin numbers of a smooth manifold that satisfy the Riemann--Roch integrality relations, these relations were computed explicitly according to the Hattori--Stong Theorem that characterizes the smooth ($SO$) cobordism group. Combining the relations with the signature equation, the topological existence problem was reduced to a number--theoretic question asking whether there exist solutions to a system of Diophantine equations with rational coefficients involving Bernoulli numbers.

In \cite{Su09} and \cite{Su14}, the author showed that above dimension 4, 8, 16, in which the well-known complex, quaternionic, octonionic projective planes ($\CP^2, \HP^2, \OP^2$) exist respectively,  the smallest possible dimension a rational projective plane can exist is 32. In later papers \cite{FS16}, \cite{KS17}, the author and her collaborators Jim Fowler and Lee Kennard showed that (1) a rational projective plane of dimension greater than 4 can only exist in dimensions $n=8k$ where $k=2^a$ or $k=2^a+2^b$; (2) such manifold exist in dimension $n\leq 512$ if and only if $n\in\{4,8,16,32,128, 256\}$, and does not exist in any dimension $512 < n < 2^{13}$ except for four possible exceptions; (3) such manifold can not admit Spin structure if the dimension $n>16$. In 2017, Matthias Kreck suggested to the author that the rational projective plane problem is an example of the more general question of realizing prescribed Betti numbers, in light of which we phrase the question being studied in this paper as in the title.

In \cite{AM19}, Albanese and Milivojevi\'{c} showed that a closed almost complex manifold with sum of Betti numbers three can only exist in dimension $n=2^a$. Note that a closed smooth manifold with sum of Betti number three is exactly what we called a rational projective plane, existence of such smooth manifold was only confirmed in dimension $n=2^a$ with $a\in\{2,3,4,5,7, 8\}$. In \cite{FS16} we showed that the signature equation written in terms of the Pontryagin numbers $p_{n/4}$ and $p_{n/8}^2$ can obstruct  existence of such smooth manifold in any dimension other than $n=8(2^a)$ or $n=8(2^a+2^b)$ when $n> 4$, while Albanese and Milivojevi\'{c} expressed the Pontryagin numbers in terms of the Chern numbers $c_{n/2}$ and $c_{n/4}^2$ in the signature equation to obtain their result on almost complex manifold. 

In order to conclude if there exists any almost complex rational projective plane in dimension $n=2^a$,  the approach that worked for smooth projective plane suggests combining the signature equation with the Riemann--Roch integrality relations among Chern numbers of almost complex manifolds. Following up Albanese and Milivojevi\'{c}'s result published on arXiv in 2018, the author adopted one particular integrality condition of the Chern number $c_{n/4}^2$ according to Stong's theorem characterizing the almost complex ($U$) cobordism, it turned out that this condition combined with the signature equation can rule out the existence of almost complex rational projective planes in any dimension $n>4$ (see Corollary \ref{nonexist-2}), i.e.,there does not exist any closed almost complex manifold with sum of Betti numbers three above dimension 4. The author did not publish this result but had announced it in a conference in summer 2018.

Motivated by the negative result in realizing middle Betti number $b_{n/2}=1$, this paper studies existence of $n$--dimensional almost complex manifold with Betti number $b_i=0$ except $b_0=b_n=1$ and $b_{n/2}\geq 1$. Parallel to the smooth manifold realization problem, we apply the version of Sullivan's theorem for almost complex manifold.  In 2019, the author was informed by Sullivan that his student Jiahao Hu had also obtained the nonexistence result of almost complex manifold with sum of Betti numbers three (\cite{JH21}), and his student Aleksandar Milivojevi\'{c} will be writing his thesis on the version of rational surgery realization theorem on almost complex manifolds. As both parties had already obtained partial results at the time, the author and Milivojevi\'{c} communicated with each other and shared ideas on the realization theorem and its application. The author thank Sullivan and Milivojevi\'{c} for their communication, in particular, for pointing out that in the proof of the theorem, the pullback space for surgery is non-simply-connected when the input map $X\to BU_0$ has vanishing $c_1$ class. Milivojevi\'{c}'s PhD thesis \cite{Mil21} gave an exposition of Sullivan's theorem on realizing rational homotopy types by closed almost complex manifold, including a careful explanation on how surgery was performed on a rational degree one normal map to obtain a rational homotopy equivalence. In this paper, we phrase Sullivan's theorem by stating the necessary and sufficient conditions for realizing a rational cohomology ring by simply-connected closed almost complex manifold, followed by a proof parallel to the one given in the author's thesis for smooth manifold, we explicitly compute the conditions under our Betti number assumption for any dimension $n=4k$. The author would like to thank Jim Davis for some helpful discussions.

In section \ref{sec:2}, we state Sullivan's theorem on realizing a rational cohomology ring by $n=4k$ dimensional simply-connected closed almost complex manifold.  Compared to the original theorem for smooth manifold, the necessary and sufficient conditions are expressed in terms of Chern numbers instead of Pontryagin numbers. Parallel to Hattori-Stong integrality relations among Pontryagin numbers, the Riemann--Roch integrality relations among Chern numbers of almost complex manifolds are expressed according to Stong's Theorem characterizing $U$ and $SU$ cobordism. Depending on the first Chern class being non-zero or zero, the input Chern classes are either maps from the rational space $X$ to the classifying space $BU$ or $BSU$ respectively.

In section \ref{sec:3}, we show that under the assumption when all Chern classes vanish except possibly for the middle and top dimensional classes $c_{n/4}$ and $c_{n/2}$, the Riemann--Roch integrality relations among Chern numbers of $n=4k$--dimensional almost complex manifold are equivalent to a set of three sub-relations, the rational coefficients of these relations are computed explicitly. We also compute the additional relations among Chern numbers of $SU$ manifold when the dimension $n\equiv 4\pmod{8}$, .

In section \ref{sec:4}, we apply the realization theorem to formulate the necessary and sufficient conditions for a rational cohomology ring under our Betti number assumption to be realized by a simply-connected almost complex manifold. In dimension $n=8k$ and $n=8k+4$ respectively, we characterize all  realizable rational cohomology rings by a set of congruence relations among the signature and Euler characteristic. Examples in dimensions 8, 12, and 16 are computed explicitly. Then we derive the following divisibility statements on the 2-adic orders of the signature and Euler characteristic.

\begin{maintheorem}[Theorem 11 and 12]\label{sigma-euler-div-8k}

If $M^n$ is an $n=8k$--dimensional almost complex manifold with Betti number $b_i=0$ except $b_0=b_n=1$ and $b_{n/2}\geq 1$, the 2--adic order of its signature $\sigma$ and Euler characteristic $\chi$ must satisfy \begin{eqnarray} 
&&\nu_2(\sigma) \geq 4k - 2\nu_2(k) - 3, \label{sigma-div-8k}\nonumber\\
&&\nu_2(\chi) \geq 4k - 2\nu_2(k) - 2\textup{wt}(k)-2. \label{euler-div-8k}\nonumber
\end{eqnarray}

\begin{small}For example, in dimension $n=8$, $\sigma\equiv 0\bmod{2}$,

\hspace{1.75cm} in dimension $n=16$, $\sigma\equiv 0\bmod{2^3}$ and $\chi\equiv 0\bmod{2^2}$,

\hspace{1.75cm} in dimension $n=24$, $\sigma\equiv 0\bmod{2^9}$ and $\chi\equiv 0\bmod{2^6}$.
\end{small}

If $M^n$ is an $n=8k+4 (k>0)$--dimensional almost complex manifold with Betti number $b_i=0$ except $b_0=b_n=1$ and $b_{n/2}\geq 1$, 
\begin{eqnarray}
&&\nu_2(\sigma)\geq 4k+4 \nonumber\label{sigma-div-8k+4.}\\
&&\nu_2(\chi)\geq 4k-2\textup{wt}(k)\nonumber \label{euler-div-8k+4.}
\end{eqnarray}

\end{maintheorem}

\begin{maintheorem}[Theorem 13 and Corollary 14]\label{nonexist 1 and 2} 
An $n=4k (k>1)$--dimensional closed almost complex manifold with Betti number $b_i=0$ except $b_0=b_n=1$,$b_{n/2}\geq 1$ must have even signature $\sigma$ and even Euler characteristic $\chi$, i.e., the middle Betti number $b_{n/2}$ must be even.

In dimension greater than 4, the rarely existing rational projective plane \textup{(}smooth manifold whose Betti number $b_0=b_{n/2}=b_n=1$\textup{)} does not admit any almost complex structure.  Equivalently, there does not exist any closed almost complex manifold whose sum of Betti numbers equals three. 
\end{maintheorem}

\section{The rational surgery realization theorem for almost complex manifold}
\label{sec:2}

Firstly, we give a brief review of the original rational surgery realization theorem on smooth manifold (\cite[Theorem 13.2]{Sullivan77}, \cite{Barge76}). 
If a 1-connected rational Poincar\'{e} duality algebra $H^*$ can be realized as the rational cohomology ring  of an $n=4k$--dimensional closed smooth manifold, it must satisfy the following necessary conditions. The bilinear form $(H^{2k},\lambda)$, being induced from the intersection form of a closed manifold, must be isomorphic to an orthogonal sum of copies of $\langle 1\rangle$ and $\langle -1\rangle$ over $\Q$. The ring $H^*$ must contain cohomology classes $p_1,\ldots,p_k$, so that the pairing of their products of total degree $n$ with a fundamental homology class $\mu$ agree with the Pontryagin numbers of a smooth manifold, so the numbers must satisfy  the Hirzebruch signature equation and the Riemann-Roch Integrality relations which characterize the smooth ($SO$) cobordism. Sullivan's theorem showed that these necessary conditions are sufficient to construct a simply-connected closed smooth manifold realizing $H^*$ using rational surgery. The input data for rational surgery is a $\Q$--local $\Q$ Poincar\'{e} duality space $X^n$ with cohomology classes $p_i\in H^{4i}(X;\Q)$ for $i=1,\ldots, k$ and a fundamental class $\mu\in H_n(X;\Z)\cong \Q$. The pullback digram of the classifying map $p: X \xrightarrow{(p_1,\ldots,p_k)} BSO_{0}\simeq \prod\limits K(\mathbb{Q},4i)$ and the localization $BSO\to BSO_{(0)}$ gives rise to a normal map $f: M\to X$ such that $f^*(p_i)=p_i(\tau_M)$. If the pairings $\langle p_I, \mu\rangle$ satisfy the Riemann-Roch Integrality relations, a normal map of degree one in the sense that $f_*[M]=\mu$ is guaranteed to exist. Then surgery can be performed on the candidate normal map to obtain a rational homotopy equivalence if the intersection form condition and the signature equation are satisfied to give a vanishing surgery obstruction. 

Analogous to the smooth case, if a 1-connected rational Poincar\'{e} duality algebra $H^*$ can be realized as the rational cohomology ring  of a $n=4k$--dimensional closed almost complex manifold, besides the condition on the intersection form, the ring $H^*$ must contain cohomology classes $c_1,\ldots,c_{2k}$, so that the pairing of their products of total degree $n$ with a fundamental homology class $\mu$ agree with the Chern numbers of an almost complex manifold, therefore the numbers must satisfy the signature equation and the Riemann-Roch Integrality relations which characterizes the almost complex ($U$) cobordism, as well as the relations for $SU$ cobordism if $c_1=0$. Moreover, the number $\langle c_{n/2}, \mu\rangle$ should match with the Euler characteristic of $H^*$.  The following version of Sullivan's theorem states that these necessary conditions are also sufficient for the realization. 

\begin{theorem}\label{acthm} Let $H^*=\bigoplus_{i=0}^n H^i$ be a $1$--connected $\Q$ Poincar\'{e} duality algebra of dimension $n=4k$ with $k>1$. There exists $n=4k (k>1)$--dimensional, simply-connected, closed, almost complex manifold $M$ such that $H^*(M;\Q)\cong H^*$ if and only if there exist choice of total class $c=1+c_1+c_2+\cdots+c_{2k} \in \bigoplus_{i=0}^{2k} H^{2i}$, and a fundamental class $\mu \in H_0=(H^n)^*\cong \Q$ that satisfy the following conditions:

\begin{enumerate}[label={\upshape(\roman*)}]
\item The bilinear form $\lambda_{\mu}: H^{2k}\times
H^{2k}\rightarrow \Q$ defined as $(\alpha, \beta)\mapsto\langle\alpha\beta,\mu\rangle$ is isomorphic to an orthogonal sum of copies of $\langle1\rangle$ and $\langle-1\rangle$.

\item
Let $p_i(c)=(-1)^i\sum_{j=0}^{2i}(-1)^{j}c_jc_{2i-j}$, the $k$--th $\LP$--polynomial of $p_i(c)$ paired with $\mu$ equals the signature of the bilinear form $(H^{2k},\lambda_{\mu})$,
\begin{equation}\label{signature-eq}\langle \LP_k (p_1(c), \ldots p_k(c)), \mu\rangle=\sigma(H^{2k}, \lambda_{\mu})
\end{equation}

\item
The pairings $\langle c_I, \mu\rangle=\langle c_{i_1}\cdots c_{i_r}, \mu\rangle$ over all the partitions $I=(i_1,\ldots,i_r)$ of $2k$ agree with the Chern numbers of an almost complex manifold, i.e., there exists a $4k$-dimensional almost complex manifold $N$ such that $\langle c_I(\tau_N),[N]\rangle=\langle c_I, \mu\rangle$ for all partitions $I$ of $2k$. By \cite[Theorem 1]{Stong65I}, this is equivalent to say the numbers $\langle c_I, \mu\rangle$ must satisfy the integrality relations among Chern numbers of almost complex manifold given by the Hirzebruch-Riemann-Roch Theorem, i.e.,
\begin{equation}\label{BUrelation}\langle \Z[e^c_1, e^c_2, \ldots] \mathord{\cdot} \td, \mu\rangle\in \Z \end{equation}
where $e^c_l$ is the $l$-th elementary symmetric polynomial of the variables $e^{x_i}-1=\sum_{m=1}^{\infty}\frac{x_i^m}{m!}$. By the formal expression $c= \prod_{i}(1+x_i)$, each $e^c_l$ can be expressed as a polynomial in $c_1, c_2, \ldots, c_{2k}$ with rational coefficients, 
$$e^c_l:=\sigma_l(e^{x_1}-1, e^{x_2}-1, \cdots)=c_l+\textup{higher degree terms}.$$

 The Todd class $\td$ is expressed as a polynomial in the $c_i$ classes:
$$\td=\prod_i\df{x_i}{1-e^{-x_i}}=1+\frac{c_1}{2}+\frac{c_1^2+c_2}{12}+\textup{higher degree terms}$$

\item 

If $n\equiv 4 \pmod{8}$,

Case 1. If $c_1\neq 0$, no additional condition besides \textup{(i)-(iv)}.

Case 2. If $c_1=0$, the parings $\langle c_I, \mu\rangle$ over all the partitions $I$ of $2k$ agree with the Chern numbers of a $SU$ manifold, i.e., there exists a $4k$-dimensional $SU$ manifold $N$ such that $\langle c_I(\tau_N),[N]\rangle=\langle c_I, \mu\rangle$ for all partitions $I$ of $2k$. By \cite[Theorem 1]{Stong65II} and \cite[Page 259]{Stong68}, this is equivalent to say the numbers $\langle c_I, \mu\rangle$ 
must satisfy the integrality relations among Chern numbers of SU manifold given by the Riemann-Roch Theorem, i.e., in additional to the relations \eqref{BUrelation} in \textup{(iii)}, 
\begin{equation}\label{BSUrelation4mod8}\langle\Z[e^{p(c)}_1, e^{p(c)}_2, \ldots] \mathord{\cdot} \AP, \mu\rangle\in 2\Z\end{equation}
where $e^p_l(c)$ is the $l$-th elementary symmetric polynomial of the variables $e^{x_i}+e^{-x_i}-2=2\sum_{m=1}^{\infty}\frac{x_i^{2m}}{(2m)!}$. By the formal expression $c= \prod_i(1+x_i)$ and $p(c)=\prod_i(1+x_i^2)$, each $e^{p(c)}_l$ can be expressed as a polynomial in $c_1, c_2, \ldots, c_{2k}$ with rational coefficients, 
$$e^{p(c)}_l:=\sigma_l(e^{x_1}+e^{-x_1}-2, e^{x_{2}}+e^{-x_{2}}-2, \cdots)=p_l(c)+\textup{higher degree terms}.$$
The A--hat class $\AP$ is expressed as a polynomial in the $c_i$ classes:
$$\AP=\prod_i\df{x_i}{2\sinh(x_i/2)}=1-\frac{p_1(c)}{24}+\frac{7p_1(c)^2+4p_2(c)}{5760}+\textup{higher degree terms}.$$

\item 
The number $\langle c_{2k}, \mu\rangle=\chi(H^*) = \sum_{i=0}^n (-1)^i\textup{rk}(H^i)$, the Euler characteristic of $H^*$.

\end{enumerate}

Any choice of $c$ and $\mu$ satisfying the conditions corresponds to a realizing manifold $M$ whose Chern numbers $\langle c_I(\tau_M),[M]\rangle=\langle c_I, \mu\rangle$. In particular, if the choice has $c_1=0$, $H^*$ can be realized by an almost complex manifold $M$ with integral first Chern class $c_1(\tau_M)=0$, i.e., $H^*$ can be realized by a SU manifold; if the choice has $c_1\neq 0$, $H^*$ can be realized by an almost complex manifold $M$ with integral first Chern class $c_1(\tau_M)\neq 0$.\\

\end{theorem}

\begin{proof} ($\Longrightarrow$) 
It is mostly straightforward that the conditions are necessary. If $M$ is a simply-connected almost complex manifold such that $H^*(M;\Q)\cong H^*$, let $c$ be the rational total Chern class $c(\tau_M)\in \bigoplus_{i=0}^{2k} H^{2i}(M;\Q)$, let $\mu$ be the image of the fundamental class $[M]$ in $H_{n}(M;\Q)$,

(i) The bilinear form $(H^{2k}, \lambda_\mu)$ is isomorphic to the rational intersection form of $M$, since the intersection form of a $4k$--dimensional manifold is a symmetric inner product space over $\Z$, $(H^{2k}, \lambda_\mu)$ is in the image of $W(\Z)\to W(\Q)$, which consists exactly of orthogonal sum of copies of $\langle 1\rangle$ and $\langle -1\rangle$ (\cite[IV. 2.6]{MilnorHusemoller73}).

(ii) The Pontryagin classes of the tangent bundle of $M$ can be written in terms of the Chern classes by the formula $p_i(\tau_M)=(-1)^i\sum_{j=0}^{2i}(-1)^{j}c_j(\tau_ M)c_{2i-j}(\tau_M)$. By the Hirzebruch signature theorem, $\langle \LP_k (p_1(c), \ldots p_k(c)), \mu\rangle=\langle \LP_k (p_1(\tau_M), \ldots p_k(\tau_M)), [M]\rangle=\sigma(M)=\sigma(H^{2k}, \lambda_{\mu})$.

(iii) For all partitions $I=(i_1,i_2, \ldots, i_r)$ of $2k$, the numbers $\langle c_I, \mu\rangle=\langle c_I(\tau_M), [M]\rangle$ must satisfy the Riemann--Roch integrality relations among Chern numbers of almost complex manifold.  By Stong's theorem (\cite[Theorem 1]{Stong65I}), these relations can be formulated as \eqref{BUrelation}.

(iv) To see {\it Case 2}, note that the realizing almost complex manifold $M$ is assumed to be simply-connected, since $H^2(M;\Z)$ is free, its integral first Chern class is zero if and only if the rational first Chern class is zero. If $M$ has the first Chern class $c_1(\tau_M)=0$, $M$ is a $SU$ manifold, the numbers $\langle c_I, \mu\rangle=\langle c_I(\tau_M), [M]\rangle$ satisfy the Riemann--Roch integrality relations among Chern numbers of $SU$ manifold. By Stong's theorem (\cite[Theorem 1(a)]{Stong65II}, \cite[Page 259]{Stong68}), if $n\not\equiv 4 \pmod{8}$,  the $SU$ relations agree with the $U$ relations \eqref{BUrelation} . If $n\equiv 4 \pmod{8}$, the $SU$ relations include both \eqref{BUrelation} and \eqref{BSUrelation4mod8}. The proof for the sufficiency of the conditions provides some further explanations on these integrality relations.

(v) The top Chern class $c_{2k}(\tau_M)$, i.e., the Euler class, paired with $[M]$ is equal to the Euler characteristic of $M$, so $\langle  c_{2k}, \mu\rangle =\langle  c_{2k}(\tau_M), [M]\rangle=\chi(M)= \chi(H^*)$.\\

($\Longleftarrow$) By \cite[page 210, Theorem 1]{Quillen69}, if $H^*=\bigoplus_{i=0}^n H^i$ is a graded commutative algebra over $\Q$ with $H^i$ finite dimensional for each $i$ and $H^1=0$, $H^0=\Q$, then $H^*$ is isomorphic to the rational cohomology ring of a simply connected space. Let $X$ be a $n=4k$--dimensional simply-connected $\Q$--local (rational) space such that $H^*(X; \Q)\cong H^*$. Assume there exist total cohomology class $c=1+c_1+c_2+\cdots+c_{2k} \in \bigoplus_{i=0}^{2k} H^{2i}(X;\Q)$ and a fundamental class $\mu \in H_n(X;\Z)\cong H_n(X;\Q)\cong \Q$, we apply rational surgery to show the existence of an $n=4k (k>1)$--dimensional almost complex manifold $M^n$ and a $\Q$--homotopy equivalence $f:M\to X$ such that $f^*(c_i)=c_i(\tau_M)$ and $f_*[M]=\mu$.

We state a proof in a similar fashion to the one given in the author's thesis \cite{Su09} (and \cite{Su14}) for Sullivan's theorem on smooth manifold. Any choice of cohomology class $c=1+c_1\cdots+c_{2k} \in \bigoplus_{i=0}^{2k} H^{2i}(X;\Q)$ corresponds to a map
$$c: X \xrightarrow{(c_1,\ldots,c_{2k}, 0,\ldots)} \prod_{i=1}^m K(\mathbb{Q},2i)\simeq BU(m)_{0}$$
where $BU(m)_{0}$ is the rationalization (localization) of $BU(m)=G_m(\C^{\infty})$ and $m>>n$.  Let $\gamma^m_{\C}=\gamma^m(\C^{\infty})$ denote the universal complex $m$--plane bundle, which is a real $2m$--plane bundle, define the map 
\[\overline{c}(\gamma):BU(m)\xrightarrow{(\overline{c}_1(\gamma^m_{\C}),\ldots,\overline{c}_{m}(\gamma^m_{\C}))}
 \prod\limits_{i=1}^m K(\mathbb{Q},2i)\simeq BU(m)_{0},\]
 where $\overline{c}(\gamma^m_{\C})=1+\overline{c}_1(\gamma^m_{\C})+\cdots+\overline{c}_m(\gamma^m_{\C})+\cdots\in \prod H^{2i}(BU(m);\Q)$ is the inverse class determined algebraically by $c(\gamma^m_{\C})\overline{c}(\gamma^m_{\C})=1$. The map $\overline{c}(\gamma)$ induces isomorphism on $\pi_*(-)\otimes \Q$ and $H_*(-; \Q)$. As shown in the diagram below, let $PB$ be the homotopy pullback space of $c$ and
$\overline{c}(\gamma)$, let $\xi^m_{\C}=pr_2^*(\gamma^m_{\C})$ be the induced bundle over $PB$. The map $pr_1: PB\to X$ induces isomorphism on $\pi_*(-)\otimes \Q$ for $*>1$. 
\begin{small}\[\xymatrix{
\nu_{M}^m \ar[d] \ar[rr]^{\widehat{g}}  && \xi^m_{\C} \ar[d] \ar[r]^{\widehat{pr_2}} & \gamma^m_{\C} \ar[d]\\  
M \ar[rrd]_f \ar[rr]^{g} &&PB \ar[d]^{pr_1} \ar[r]^-{pr_2} & BU(m)
\ar[d(0.88)]^{\overline{c}(\gamma)} \\
 && X \ar[r]^-{c} & \prod\limits_{i=1}^m K(\mathbb{Q},2i)\simeq BU(m)_{0} 
 }\]\end{small}

We would like to construct a normal map $(g, \widehat{g}): (M, \nu_M)\to (PB, \xi)$ for simply-connected surgery. In the original realization theorem on smooth manifold, the pullback space $PB$ is automatically simply-connected for any input map $p: X \to\prod\limits K(\mathbb{Q},4i)\simeq  BSO_{0}$. This is not the case for almost complex manifold. This observation and the following argument was due to Aleksandar Milivojevi\'{c} and was communicated to the author in 2019. Let $F$ denote the homotopy fiber of the map $\overline{c}: BU\to BU_0$, the exact sequence below shows that $PB$ is simply connected if and only if the map $\pi_2(X)\to \pi_2(BU_0)$ is surjective, which happens if and only if the input class $c_1\in H^2(X;\Q)$ is nonzero.\\

\begin{small}$\xymatrix{
\ar[r] &\pi_2(PB) \ar[d] \ar[r] & \pi_2(X) \ar[d] \ar[r] &\pi_1(F)\ar[d] \ar[r] & \pi_1(PB)\ar[d] \ar[r] &\pi_1(X)=0\ar[d]  \\
\ar[r] &\pi_2(BU)  \ar[r] & \pi_2(BU_{0}) \ar[r] &\pi_1(F) \ar[r] & \pi_1(BU) \ar[r] &\pi_1(BU_{0})
}$

$\hspace{1.6cm}\Z\ \ \ \ \ \ \ \ \ \ \ \ \ \ \ \ \ \Q\ \ \ \ \ \ \ \ \ \ \ \ \ \ \ \Q/\Z \ \ \ \ \ \ \ \ \ \ \ \ \ \ 0$\\

\end{small}

The proof is then divided into two cases. If $c_1$ is nonzero, we show that the above pullback diagram with classifying space $BU$ yields a normal map $f:M\to X$ where $M$ is a stably almost complex manifold, the existence of a degree 1 normal map is guaranteed by condition (iii). If $c_1$ is zero, the classifying space $BU$ is replaced by $BSU$ in the pullback diagram, which yields a normal map $f:M\to X$ where $M$ is a stably almost complex manifold with $c_1(\tau_M)=0$, in this case, condition (iii) and (iv) Case 2 are sufficient to produce a degree 1 normal map.

{\it Case} 1. If $c_1\neq 0$, the pullback space $PB$ is simply-connected. The generalized Thom-Pontryagin construction (\cite[page 264--266]{Lashof63} and \cite[page 18--23]{Stong68}) produces a surgery normal map as follows. Let $\alpha$ be any class in $\pi_{n+2m}(T\xi^m_{\C})$, for the composite $S^{n+2m}\xrightarrow{\alpha} T\xi^m_{\C}\xrightarrow{Tpr_2} T\gamma_{\C}^m$, since $S^{n+2m}$ is compact, the image $Tpr_2\circ\alpha(S^{n+2m})$ lives in $T\gamma^m(\C^{r})$ for some finite $r$. By the transversality theorem, $Tpr_2\circ \alpha$ can be deformed to a map $h$ that is differentiable in the pre-image of some open neighborhood of $Gr_m({\C}^r)\subset T\gamma^m(\C^{r})$ and is transverse regular on $Gr_m({\C}^r)$. Then we obtain a manifold $M^{n}=h^{-1}(\text{Gr}_m({\C}^r))$ and $h |_{M}$ is the classifying map of the stable normal bundle $\nu_M$. A complex structure on the stable normal bundle of $M$ determines a complex structure on the stable tangent bundle of $M$ up to homotopy, so $M$ is a stably almost complex manifold.  By the homotopy lifting property, the deformation of $Tpr_2\circ \alpha$ to $h$ can be covered by a homotopy of $\alpha$ to a new map $g$ such that $h=Tpr_2\circ g$, then $M=g^{-1}(PB)=h^{-1}(Gr_m({\C}^r))$. We have obtained a normal map $(g, \widehat{g}): (M, \nu_M)\to (PB, \xi)$. Chasing the diagram, for the map $f:=pr_1\circ g: M\rightarrow X$, $f^*(c_i)=c_i(\tau_{M}) \in H^{2i}(M;\Q)$, which is the rational Chern class of the stable tangent bundle of $M$.

The existence of a degree 1 normal map in the sense that $f_*[M]=\mu$ is guaranteed by condition (iii). We outline a proof similar to the argument given in \cite[Lemma 3.2.2.]{Su09} and \cite{Su14} for the theorem on smooth manifold. If the numbers $\langle c_I, \mu\rangle$ agree with the set of Chern numbers of an almost complex manifold, i.e., there exists a $n$--dimensional almost complex manifold $N$ such that for all partitions $I$ of $2k$,
$$\langle c_I(\tau_N),[N]\rangle=\langle c_I,\mu\rangle=\langle c_I(\gamma^m_{\C}),c_*\mu\rangle=\langle\overline{c}_I(\gamma^m_{\C}),{\overline{c}(\gamma)}_*^{-1}(c_*\mu)\rangle.$$
This identity implies that $\langle c_I(\nu_N),[N]\rangle=\langle
c_I(\gamma^m_{\C}),{\overline{c}(\gamma)}_*^{-1}(c_*\mu)\rangle$, 
therefore the rational homology class ${\overline{c}(\gamma)}_*^{-1}(c_*\mu)$ lies in the image of the homomorphism $\nu: \Omega_n^{U}\rightarrow H_n(BU;\mathbb{Q})$ defined by $\nu([M])=(\nu_M)_*([M])$. Since the homomorphism $\nu$ is obtainable as the composition $\Omega_n^{U}\cong \pi_{n+2m}(T\gamma^m)\xrightarrow{h}H_{n+2m}(T\gamma^m)\xrightarrow{\cap U} H_n(BU;\Z)\to H_n(BU;\Q)$, which is exactly the diagonal map in the upper right corner of the digram below, there must exist a homotopy class $\beta\in \pi_{n+2m}(T\gamma^m)=\pi_{n+2m }(MU(m))$ mapping to $c_*\mu\in H_n(BU(m)_0;\Q)$. Note that the diagonal maps in the lower left and lower right corner of the diagram are isomorphisms, this is because both the Thom space of the rational spherical fibration and the base space are $\Q$--local (rational), and the Hurewicz map on the Thom space is a rational isomorphism as we assumed $m>>n$. It can be shown that the outer square of Thom spaces is a homotopy Cartesian square (see \cite[Lemma 6.1]{TaylorWilliams}, \cite[Theorem 2.3]{Quillen69}, or a detailed proof in \cite[Lemma 3.2.3]{Su09}). Then the existence of homotopy classes $\beta\in  \pi_{n+2m}(T\gamma^m)$ and $c_X\in\pi_{n+2m}(T\widetilde{\nu}_X)$ both mapping to $c_*\mu$ implies the existence of a desired Spivak class $\alpha\in\pi_{n+2m}(T\xi^m_{\C})$ in the upper left corner mapping to $\mu$. Therefore $f_*[M]=\mu$.

\begin{small}
\[
\xymatrix @R=6mm @C=3.5mm {
\alpha\in\pi_{n+2m}(T\xi^m_{\C}) \ar[ddd]^{{Tpr_1}_*}  \ar[dr]^{} \ar[rrr]^{{T_{pr_2}}_*} & & &  \pi_{n+2m}(T\gamma^m_{\C})\ni \beta \ar[dl]^{\nu} \ar[ddd]^{T\overline{c}(\gamma)_*}\\
& H_n(PB) \ar[d]^{{pr_1}_*} \ar[r]^-{{pr_2}_*}  & H_n(BU(m)) \ar[d]^{\overline{c}(\gamma)_*} \\
& \mu\in H_n(X) \ar[r]^-{c_*}  &  H_n(BU(m)_{0}) \\
c_X\in\pi_{n+2m}(T\widetilde{\nu}_X) \ar[ur] ^{\cong}\ar[rrr]^{{Tc}_*}  & & &
\pi_{n+2m}(T\gamma^m_{0})  \ar[ul]_{\cong} }\]\\
\end{small}

{\it Case} 2. If $c_1= 0$, replace the classifying space $BU$ by $BSU$, the input map becomes $c: X \xrightarrow{(c_1=0, c_2,\ldots,c_{2k}, 0,\ldots)} \prod_{i=2}^m K(\mathbb{Q},2i)\simeq BSU(m)_{0}$. Since $\pi_2(BSU)=0$, the pullback space $PB$ is simply-connected. Let $\eta_{\C}^m$ denote the universal bundle over $BSU(m)$. 
\begin{small}\[\xymatrix @R=8mm @C=3.5mm {
\nu_{M}^m \ar[d] \ar[rr]^{\widehat{g}}  && \xi^m_{\C} \ar[d] \ar[r]^{\widehat{pr_2}} & \eta^m_{\C} \ar[d] \ar[r]& \gamma^m_{\C}\ar[d] \\  
M \ar[rrd]_f \ar[rr]^{g} &&PB \ar[d]^{pr_1} \ar[r]^-{pr_2} & BSU(m)
\ar[d(0.88)]^{\overline{c}(\gamma)}  \ar[r]& BU(m)\\
 && X \ar[r]^-{c} &\prod\limits_{i=2}^m K(\mathbb{Q},2i)\simeq BSU(m)_{0} 
& }\]\end{small}

Similar to the argument in {\it Case} 1, the Thom-Pontryagin construction produces a normal map $(g, \widehat{g}): (M, \nu_M)\to (PB, \xi)$ such that $M$ is a stably almost complex manifold and $f^*(c_i)=c_i(\tau_M)\in H^{2i}(M;\Q)$, in particular, $c_1(\tau_M)=f^*(c_1)=0$. If the numbers $\langle c_I, \mu\rangle$ agree with the set of Chern numbers of a $SU$ manifold, there exists a homotopy class $\beta\in \pi_{n+2m}(T\eta^m)=\pi_{n+2m }(MSU(m))\cong \Omega_n^{SU}$ mapping to $c_*\mu\in H_n(BSU(m)_0;\Q)$, this guarantees the existence of a Spivak class $\alpha\in\pi_{n+2m}(T\eta^m_{\C})$ mapping to $\mu$, which implies that $f_*[M]=\mu$.  

We elaborate on why the numbers $\langle c_I,\mu\rangle$ agree with the set of Chern numbers of an almost complex ($U$) or $SU$ manifold if and only if they satisfy the integrality relations phrased as \eqref{BUrelation} and \eqref{BSUrelation4mod8} in condition (iii) and (iv). For $G=U$ or $SU$, since $H^*(BG;\Q)$ is the rational polynomial algebra on the universal Chern classes and $H_n(BG;\Q)\cong Hom(H^n(BG;\Q),\Q)$, the integrality relations among Chern numbers of $G$ manifold characterize the group $\tau\Omega^{G}_n$, which is the image of the $G$ cobordism group under the homomorphism $\tau: \Omega_n^{G}\to H_n(BG; \Q)$ induced by the stable tangent bundle $\tau_M: M\to BG$. 

As stated in \cite[Theorem 1]{Stong65I}, $\tau\Omega^U_n=\{x \in H_n(BU;\Q)\,|\, \langle \Z[e^c_1, e^c_2, \ldots] \mathord{\cdot} \td, x\rangle\in \Z\}$, therefore the $U$ relations are phrased as \eqref{BUrelation} in (iii). 

In \cite[Theorem 1]{Stong65II}, \cite[Page 259]{Stong68}, the $SU$ relations are described in terms of the $s_{\omega}$--symmetric polynomials, since the $s_{\omega}$--symmetric polynomials form an additive basis for the ring of formal power series of elementary symmetric polynomials. It is equivalent to phrase the relations  in terms of  the elementary symmetric polynomials as follows.

If $n\not\equiv 4 \pmod{8}$,   $\tau\Omega^{SU}_{n}=\{x \in H_n(BSU;\Q)\,|\, \langle \Z[e^c_1, e^c_2, \ldots] \mathord{\cdot} \td, x\rangle\in \Z\},$ i.e., the $SU$ relations agree with the $U$ relations. Therefore when $c_1=0$ and $n\not\equiv 4 \pmod{8}$, the relations \eqref{BUrelation} in condition (iii) is sufficient to produce a desired normal map.

If $n\equiv 4 \pmod{8}$, $\tau\Omega^{SU}_{n}\!\!\!=\!\!\{x\! \in \!H_n(BSU;\Q)\,|\, \langle \Z[e^c_1, e^c_2, \ldots] \mathord{\cdot} \td, x\rangle\!\in \!\Z \textup{ and } \langle \Z[e^{p(c)}_1\!\!,  e^{p(c)}_2\!\!\!, \ldots] \mathord{\cdot} \AP, x\rangle\!\!\in \!\!2\Z\}.$ Therefore in (iv) Case 2, the additional relation \eqref{BSUrelation4mod8} is required.

The next step is to perform surgery on the candidate degree 1 normal map, condition (i) and (ii) guarantee a vanishing surgery obstruction so that a rational homotopy equivalence can be obtained.  The proof outlined in the original theorem \cite[Page 326]{Sullivan77} on smooth manifold still applies in the almost complex case. The book \cite{Anderson77} discussed surgery theory with $\Z_P$ coefficients ($\Z_P=\Q$ if $P$ is the set of all primes). One can also refer to Milivojevi\'{c}'s PhD thesis \cite{Mil21} for an exposition on rational surgery. 
The candidate normal map $(g, \widehat{g}): (M, \nu_M)\to (PB, \xi)$ is normally cobordant to a rational homotopy equivalence if and only if the obstruction vanishes in the $L$ group 
$L_{4k}(\Q)\cong \Z\oplus \bigoplus_{\infty}\Z_2 \oplus \bigoplus_{\infty}\Z_4$. Condition (i) guarantees that the intersection form $(H^{2k}(X;\Q),\lambda_{\mu})$ is contained in the image of the map $W(\Z)\to W(\Q)$, and therefore has a vanishing $\Z_2$ and $\Z_4$ summands in $L_{4k}(\Q)$. Since $\langle \LP_k(p_i(c)),\mu\rangle=\langle  \LP_k(p_i(c)),f_*([M])\rangle=\langle \LP_k(p_i(f^*c)),[M]\rangle=\langle \LP_k(p_i(c(\tau_M))),[M]\rangle=\sigma(M)$. The signature equation in (ii) implies that $\sigma(M)=\sigma(H^{2k}(X;\Q),\lambda_{\mu})$, so the obstruction has a vanishing $\Z$ summand in $L_{4k}(\Q)$. After a sequence of surgery performed, the resulting normal map, again denoted $(g, \widehat{g}): (M, \nu_M)\to (PB, \xi)$, still gives a complex structure on the normal bundle of $M$, the map $f: M\to X$ is a desired rational homotopy equivalence from a simply-connected stably almost complex manifold. 

Condition (v) ensures that $\langle c_{2k}(\tau_M),[M]\rangle=\langle f^*(c_{2k}),[M]\rangle  =\langle c_{2k},\mu\rangle=\chi(H^*)=\chi(X)$, which is now equal to $\chi(M)$ as $f: M\to X$ is a rational homotopy equivalence, then by \cite[Theorem 1.7]{Thomas67} and \cite[Theorem 1.1]{Sutherland65}, the $2k$--Chern number of $M$ agrees with its Euler class implies that $M$ admits an almost complex structure.

This finished the proof that conditions (i)-(v) are sufficient for the existence of a simply-connected almost complex manifold $M$ and a rational homotopy equivalence $f:M\to X$ such that $f^*(c_i)=c_i(\tau_M)$ and $f_*[M]=\mu$, and therefore $M$ realizes the prescribed rational cohomology ring $H^*$ and the Chern numbers $\langle c_I(\tau_M),[M]\rangle=\langle c_I, \mu\rangle$ for all partitions $I$ of $2k$. Note that the first rational Chern class $c_1(\tau_M)=f^*(c_1)=0$ if and only if the input class $c_1=0$. Since the realizing manifold $M$ is simply-connected, $H^2(M;\Z)$ is free, so the integral Chern class of $M$ is zero if and only if the input class $c_1=0$.  
\end{proof}

\section{Integrality relations among Chern numbers.}
\label{sec:3}

In order to apply Theorem \ref{acthm} in our prescribed Betti number setting, we explicitly calculate the integrality relations \eqref{BUrelation} and \eqref{BSUrelation4mod8} under the assumption that  all the Chern classes $c_i$ vanish except possibly for the middle and top dimensional classes $c_k$ and $c_{2k}$.

\begin{proposition}\label{U-relation-thm}

If the product of Chern classes $c_{\omega}=0$ except possibly for $c_{k}$, $c_k^2$, and $c_{2k}$, the integrality relations among Chern numbers of $n=4k (k>1)$--dimensional almost complex manifolds, expressed as \eqref{BUrelation}
\begin{equation*}\langle \Z[e^c_1, e^c_2, \ldots] \mathord{\cdot} \td, \mu\rangle\in \Z,\end{equation*}  are satisfied if and only if the following sub-relations are satisfied.
\begin{subnumcases}
\ \langle \td, \mu\rangle=\frac{1}{2}(t_k^2-t_{2k})\langle c_k^2,\mu\rangle+t_{2k}\langle c_{2k},\mu\rangle\in\Z   \ \ {\scriptstyle \textup{with}\ \   t_k=\frac{B_k}{k!}}\label{U-relation-thm-4k-Td}\\ 
\langle e_1^c\mathord{\cdot} \td, \mu\rangle=\left((-1)^{k+1}\df{t_{k}}{(k-1)!}+\df{1}{2(2k-1)!}\right)\langle c_k^2,\mu\rangle -\df{1}{(2k-1)!}\langle c_{2k},\mu\rangle\in\Z  \label{U-relation-thm-4k-e1}\\
\langle e_1^ce_1^c\mathord{\cdot} \td, \mu\rangle=\df{1}{[(k-1)!]^2}\langle c_k^2,\mu\rangle\in\delta_k\Z  \ \ \textup{with}  \ \delta_k=\left\{
                            \begin{array}{ll}
                              1  & k=2\\
                              2& k>2
                            \end{array}
                          \right.  \label{U-relation-thm-4k-e1e1}
\end{subnumcases}
\end{proposition}

\begin{proof}
The proof is similar to the one given in \cite[Lemma 2, 3, 4]{KS17}. Since $c_{\omega}=0$ except possibly for $c_{k}$, $c_k^2$, and $c_{2k}$, the total Todd class is $\td=1+\td_{k}+\td_{2k}=1+t_{k}c_{k}+(t_{k,k}c_{k}^2+t_{2k}c_{2k})$ where $t_{k}=B_{k}/k!$ and  $t_{k,k}=\frac{1}{2}(t_{k}^2-t_{2k})$.  

In our notation, $e_l^{c}$ is the $l$--th elementary symmetric polynomial of the variables $e^{x_i}-1=\sum_{m=1}^\infty \frac{x_i^{m}}{m!}$. Since each $e_l^{c}$ class can be expressed as a linear combination of $c_k$, $c_k^2$ and $c_{2k}$ with rational coefficients, any monomial $e_{\omega}^c=0$ if the degree $|\omega|>2$, so  \eqref{BUrelation} holds true if and only if $\langle \td, \mu\rangle\in\Z$, $\langle e_l^c\mathord{\cdot} \td, \mu\rangle\in\Z$, and $\langle e_l^ce_m^c\mathord{\cdot} \td, \mu\rangle\in\Z$ for all $l, m\geq 1$, we will show the later two conditions hold true if and only if \eqref{U-relation-thm-4k-e1} and \eqref{U-relation-thm-4k-e1e1} hold true. An explicit formula for the $e_l^c$ class will be derived via the power sum symmetric polynomial.

Let $S_m(x_i)$ denote the $m$-th power sum of the variables $x_i$, i.e., $S_m(x_i)=\sum_i x_i^m$. Apply the Newton-Girard formula, the power sum symmetric polynomial $S_m(x_i)$ can be written in terms of the elementary symmetric polynomials $c_j=\sigma_j(x_i)$,
\begin{small} 
\begin{equation}\label{Sx}
S_m(x_i)=\begin{cases} 
(-1)^{k+1}kc_k & \text{if } m=k \\
  k(c_k^2-2c_{2k}) & \text{if } m=2k\\
  0 & \text{otherwise}
  \end{cases}.
\end{equation}
\end{small}
Then we can write
\begin{small}
	\begin{eqnarray}\label{S1e1}
	e_1^c= \ds_{i} e^{x_i}-1 =\sum_i  \sum_{m=1}^\infty \frac{ x_i^m}{m!} &=& \sum_{m=1}^\infty \frac{1}{m!} S_m(x_i)\nonumber\\
	                  &=& \df{1}{k!}S_k(x_i)+\df{1}{(2k)!}S_{2k}(x_i)\nonumber\\
	                  &\stackrel{}{=}& \df{(-1)^{k+1}}{(k-1)!}c_k+\df{1}{2(2k-1)!}(c_k^2-2c_{2k}).
	\end{eqnarray}	
\end{small}
Let $S_l^c$ denote the $l$-th power sum of the variables $e^{x_i}-1$.  Let $\{-\}_{m}$ denote the degree $m$ terms of the variable $x_i$'s in an expression. Let $C_l(m):=\sum_{j=0}^{l-1} (-1)^j \binom{l}{j} (l-j)^{m}$.
\begin{small}
 \begin{eqnarray}\label{Sl-e1}
\{S_{l}^c\}_m=\left\{\ds_i(e^{x_i}-1)^l\right\}_m
	&=&\left\{\ds_{i}\ds_{j=0}^{l}(-1)^j\binom{l}{j}e^{x_i(l-j)}\right\}_m\nonumber\\
	&=&\ds_{i}\ds_{j=0}^{l}(-1)^j\binom{l}{j}\df{x_i^{m}(l-j)^{m}}{m!}\nonumber\\
	&=&\left[\sum_{j=0}^{l-1} (-1)^j \binom{l}{j} (l-j)^{m}\right]\frac{S_m(x_i)}{m!}= C_l(m)\{e_1^c\}_m.\nonumber
\end{eqnarray}
\end{small}
Note that $C_l(m)=l!S(m,l)$ where $S(m,l)$ is the Stirling number of the second kind. If $m<l$, $C_l(m)=0$. By \eqref{Sl-e1} and \eqref{S1e1},\begin{eqnarray}S_l^c=\{S_l^c\}_k + \{S_l^c\}_{2k}	&=& C_l(k)\{e_1^c\}_k + C_l(2k)\{e_1^c\}_{2k}\label{Sl1}\\
&=&C_l(2k)\left(\{e_1^c\}_k+\{e_1^c\}_{2k}\right) + \left[C_l(k)-C_l(2k)\right]\{e_1^c\}_k\nonumber\\
&=&C_l(2k)e_1^c+[C_l(k)-C_l(2k)]\df{(-1)^{k+1}}{(k-1)!}\,c_k. \label{Sl2}
\end{eqnarray}
Apply Newton-Girard again to relate the power sum symmetric polynomial $S_l^c=S_l^c(e^{x_i}-1)$ with the elementary symmetric polynomials $e_j^c=\sigma_j(e^{x_i}-1)$, since any monomial $e_{\omega}^c=0$ if the degree $|\omega|>2$, we can solve that\begin{small}
\begin{eqnarray}\label{el}e_l^c&=&\df{(-1)^{l+1}}{l}S_l^c+\df{1}{2}\ds_{i=1}^{l-1}e_i^ce_{l-i}^c.
\end{eqnarray}
\end{small}
Multiply $e_l^c$ with the total Todd class $\td=1+\td_{k}+\td_{2k}=1+\df{B_k}{k!}c_k+\td_{4k}$, then evaluate on the fundamental class $\mu$, by \eqref{el} and \eqref{Sl2},
\begin{small}
\begin{equation}\label{elT}\langle e_l^c\mathord{\cdot}\td,\mu\rangle=(-1)^{l+1}\underbrace{\df{C_l(2k)}{l}\langle e_1^c\mathord{\cdot}\td,\mu\rangle}_{(a)}
+(-1)^{l+k}\underbrace{{\df{[C_l(k)-C_l(2k)]}{l(k-1)!}\langle c_k\mathord{\cdot}\td,\mu\rangle}}_{(b)}
+\df{1}{2}\underbrace{\ds_{i=1}^{l-1}\langle e_i^ce_{l-i}^c\mathord{\cdot}\td,\mu\rangle}_{(c)}.
\end{equation}
\end{small}
By \eqref{el} and \eqref{Sl1},
\begin{eqnarray}\langle e_l^ce_m^c\mathord{\cdot}\td, \mu\rangle&=&(-1)^{l+m}\df{C_l(k)}{l}\df{C_m(k)}{m}\langle e_1^ce_1^c\mathord{\cdot}\td, \mu\rangle.
\end{eqnarray}

If we assume $\langle e_1^c\mathord{\cdot}\td,\mu\rangle\in\Z$ and $\langle e_1^ce_1^c\mathord{\cdot}\td,\mu\rangle\in2\Z$, then $\langle e_l^ce_m^c\mathord{\cdot}\td, \mu\rangle\in\Z$ because $C_i(k)/i$ is an integer for any $i$. In the expression of $\langle e_l^c\mathord{\cdot}\td,\mu\rangle$, $(a)\in\Z$ for the same reason. For $(b)$, since $\langle e_1^ce_1^c\mathord{\cdot}\td,\mu\rangle=\df{1}{[(k-1)!]^2}\langle c_k^2,\mu\rangle$, $\langle c_k\mathord{\cdot}\td,\mu\rangle=\df{B_{k}}{k!}[(k-1)!]^2\langle e_1^ce_1^c\mathord{\cdot}\td,\mu\rangle$,
\begin{small}
\begin{eqnarray}\label{elTb}
(b)&=&\df{[C_l(k)-C_l(2k)]}{l}\df{B_{k}}{k}\langle e_1^ce_1^c\mathord{\cdot}\td,\mu\rangle\\
&=&\sum_{j=0}^{l-1} (-1)^{j+1} \underbrace{\frac{1}{l}\binom{l}{j} (l-j)}_{(b)_1}  \underbrace{(l-j)^{k-1}\left[(l-j)^{k}-1\right]\df{B_{k}}{k}}_{(b)_2}\langle e_1^ce_1^c\mathord{\cdot}\td,\mu\rangle,\nonumber
\end{eqnarray}
\end{small}
so $(b)_1\in\Z$ from the fact that $a/\gcd(a,b)$ divides $\binom{a}{b}$ for any integer $a,b$. By a stronger version of the Lipschitz-Sylvester theorem that $a^{\lfloor \log_2 m\rfloor+1}(a^m-1)B_m/m$ is an integer for any $a\in\Z$ (\cite{Sla95}), the expression $(b)_2\in\Z$ if $k>2$. $(c)\in 2\Z$ because $\langle e_i^ce_{l-i}^c\mathord{\cdot}\td,\mu\rangle=(-1)^{l}\frac{C_i(k)}{i}\frac{C_{l-i}(k)}{l-i}\langle e_1^ce_1^c\mathord{\cdot}\td, \mu\rangle\in2\Z$ for $i=1,2\ldots l-1$. Therefore the entire expression of $\langle e_l^c\mathord{\cdot}\td,\mu\rangle\in\Z$.

If we assume $\langle e_l^c\mathord{\cdot} \td, \mu\rangle\in\Z$ and $\langle e_l^ce_m^c\mathord{\cdot} \td, \mu\rangle\in\Z $ for all $l,m\geq 1$, in the expression for 
$\langle e_2^c\mathord{\cdot}\td,\mu\rangle$, since $(a)\in\Z$, and $(b)\in\Z$ when $k>2$, $(c)/2$ must be an integer, hence $\langle e_1^ce_1^c\mathord{\cdot} \td, \mu\rangle\in2\Z$ when $k>2$.

For $k=2$,  i.e., the $8$--dimensional case, the only nontrivial $e_l^c$ classes are $e_1, e_2, e_3, e_4$. It suffices to show that $\langle e_1^c\mathord{\cdot} \td, \mu\rangle\in\Z$ and $\langle e_1^ce_1^c\mathord{\cdot} \td, \mu\rangle\in\Z$ together implies $\langle e_l^c\mathord{\cdot} \td, \mu\rangle\in\Z$ for $l=2,3,4$. By \eqref{elT}, \eqref{elTb}, \eqref{S1e1}, $\td=1+\frac{c_2}{12}+\td_4$, $C_l(2)=2, 0, 0$ and $C_l(4)=14, 36, 24$ for $l=2, 3,  4$, one can compute that
\begin{eqnarray*}
&&\langle e_1^c\mathord{\cdot} \td, \mu\rangle=\langle -\frac{c_4}{6}, \mu\rangle,  \ \ \ \  \langle e_1^ce_1^c\mathord{\cdot} \td, \mu\rangle=\langle c_2^2, \mu\rangle, \\
&&\langle e_2^c\mathord{\cdot} \td, \mu\rangle=\langle \frac{7c_4}{6}, \mu\rangle,\ \ \langle e_3^c\mathord{\cdot} \td, \mu\rangle=\langle -2c_4, \mu\rangle,\ \ \langle e_4^c\mathord{\cdot} \td, \mu\rangle=\langle c_4, \mu\rangle.
\end{eqnarray*}
Therefore the statement is true for $k=2$. Alternatively, one can verify that when $k=2$, in the expression \eqref{elT} (and \eqref{elTb}) of $\langle e_l^c\mathord{\cdot} \td, \mu\rangle$, $(b)+\df{1}{2}(c)=0$ for $l=2,3,4$, therefore $\langle e_1^c\mathord{\cdot} \td, \mu\rangle\in\Z$ itself already guarantees $\langle e_l^c\mathord{\cdot} \td, \mu\rangle\in\Z$ for all $l>1$ when $k=2$.

\end{proof}

\begin{proposition} \label{SU-relation-thm} If the product of Chern  classes $c_{\omega}=0$ except possibly for $c_{k}$, $c_k^2$, and $c_{2k}$,  the additional integrality relations \textup{(}besides \eqref{BUrelation}\textup{)} among Chern numbers of $SU$ manifolds of dimension $n=4k (k>1)$ with $k$ odd, which is expressed as \eqref{BSUrelation4mod8}
\begin{equation*}\langle\Z[e^{p(c)}_1, e^{p(c)}_2, \ldots] \mathord{\cdot} \AP, \mu\rangle\in 2\Z, 
\end{equation*} 
are satisfied if the following sub-relations are satisfied. \begin{subnumcases}
\ \langle \td, \mu\rangle=\frac{1}{2}t_{2k}[-\langle c_k^2,\mu\rangle+2\langle c_{2k},\mu\rangle]\in2\Z   \ \ {\scriptstyle \textup{with}\ \   t_{2k}=\frac{B_{2k}}{(2k)!}}\label{SU-relation-thm-Td-8k+4}\\ 
\langle e_1^c\mathord{\cdot} \td, \mu\rangle=\left(\df{1}{2(2k-1)!}\right)\langle c_k^2,\mu\rangle -\df{1}{(2k-1)!}\langle c_{2k},\mu\rangle\in\Z  \label{SU-relation-thm-e1-8k+4}
\end{subnumcases}

\end{proposition}

\begin{proof}
 $\td=\prod_i\frac{x_i}{1-e^{-x_i}}=\left(\prod_i e^{x_i/2}\right)\left(\prod_i \frac{x_i}{e^{x_i/2}-e^{-x_i/2}}\right)=e^{c_1/2}\AP$. Since $c_1=0$, the $\td$ class coincides with the $\AP$ class and the odd degree (of $c_j$) terms of $\td$ vanish for $SU$ manifold. Since $c_{\omega}=0$ except possibly for $c_{k}$, $c_k^2$, and $c_{2k}$, $k$ is odd and $t_k=\frac{B_k}{k!}=0$,
 $$\AP=\td=1+\td_{2k}=1+\frac{-1}{2}t_{2k} c_k^2+t_{2k}c_{2k}.$$
  
 In our notation, $e_l^{p(c)}$ is the $l$--th elementary symmetric polynomial of the variables $e^{x_i}+e^{-x_i}-2=\sum_{m=1}^\infty \frac{2 x_i^{2m}}{(2m)!}$, which contains only even degree terms of $x_i$, since $k$ is odd, each $e_l^{p(c)}$ class contains only degree $2k$ terms and is therefore expressed as a linear combination of $c_k^2$ and $c_{2k}$ with rational coefficients, then any monomial $e_{\omega}^{p(c)}=0$ if the degree $|\omega|>1$, so \eqref{BSUrelation4mod8} holds true if and only if $\langle \td, \mu\rangle\in2\Z$ and $\langle e_l^{p(c)}\mathord{\cdot} \td, \mu\rangle\in2\Z$ for all $l\geq 1$. We will show the almost complex relation $\langle e_1^c\mathord{\cdot}\td, \mu\rangle\in\Z$ guarantees the $SU$ relation $\langle e_l^{p(c)}\mathord{\cdot} \td, \mu\rangle\in2\Z$ for all $l\geq 1$ in this case.
 
In particular, by \eqref{Sx}, 
\begin{small}
\begin{eqnarray}\label{S1e1-SU}
	e_1^{p(c)}
	=\left\{\sum_i  \sum_{m=1}^\infty \frac{2x_i^{2m}}{(2m)!}\right\}_{2k} = \frac{2}{(2k)!} S_{2k}(x_i)=\df{1}{(2k-1)!}(c_k^2-2c_{2k}).\nonumber
	\end{eqnarray}
\end{small}
Comparing with the expression of $e_1^c$ in \eqref{S1e1}, note that $e_1^{p(c)}=2\{e_1^c\}_{2k}$.

Let $S_l^{p(c)}$ denote the $l$-th power sum of the variables $e^{x_i}+e^{-x_i}-2=\sum_{m=1}^\infty \frac{2 x_i^{2m}}{(2m)!}$, which also only contains degree $2k$ terms. Let $M_l(k):=\sum_{j=0}^{l-1} (-1)^j \binom{2l}{j} (l-j)^{2k}$.

\begin{small}
 \begin{eqnarray}\label{Sl-e1-SU}
S_{l}^{p(c)}=\ds_i(e^{x_i}+e^{-x_i}-2)^l=\ds_i\left(e^{x_i/2}-e^{-x_i/2}\right)^{2l}
\!\!\!\!&=&\!\!\!\!\left\{\ds_{i}\ds_{j=0}^{2l}(-1)^j\binom{2l}{j}e^{x_i(l-j)}\right\}_{\!\!2k}\nonumber\\
	&=&\ds_{i}\ds_{j=0}^{2l}(-1)^j\binom{2l}{j}\df{x_i^{2k}(l-j)^{2k}}{(2k)!}\nonumber\\
&=&2\left[\sum_{j=0}^{l-1} (-1)^j \binom{2l}{j} (l-j)^{2k}\right]\frac{S_{2k}(x_i)}{(2k)!} \nonumber\\
&=&M_l(k)e_1^{p(c)}.
\end{eqnarray}
\end{small}
Apply the Newton-Girard formula, since any monomial $e_{\omega}^{p(c)}=0$ if the degree $|\omega|>1$, 
\begin{eqnarray}\label{el-SU}e_l^{p(c)}&=&\df{(-1)^{l+1}}{l}S_l^{p(c)}=(-1)^{l+1}\df{M_l(k)}{l}e_1^{p(c)}.
\end{eqnarray}
By the observation that $e_1^{p(c)}=2\{e_1^c\}_{2k}$ and $\td=1+\td_{2k}$, 
\begin{equation}
\langle e_l^{p(c)}\mathord{\cdot} \td, \mu\rangle=(-1)^{l+1}\df{M_l(k)}{l}\ 2 \langle e_1^c\mathord{\cdot}\td, \mu\rangle.
\end{equation}
As addressed in \cite[Lemma 3]{KS17}, $M_l(k)$ is divisible by $l$, therefore $\langle e_1^{c}\mathord{\cdot} \td, \mu\rangle\in\Z$ implies $\langle e_l^{p(c)}\mathord{\cdot} \td, \mu\rangle\in2\Z$. 

\end{proof}

\section{almost complex manifold with Betti number $b_i=0$ except $i=0, n/2, n$.}
\label{sec:4}

We apply Theorem \ref{acthm} under our Betti number assumption. If a rational cohomology ring $H^*$ is realizable, it is necessary that condition (i) is satisfied, so we prescribe $H^*$ to have a bilinear form $(H^{2k},\lambda)$  isomorphic (over $\Q$) to an orthogonal sum of $\langle1\rangle$'s and $\langle-1\rangle$'s. Let $(H^*, \sigma, \chi)$ denote any $n=4k$--dimensional rational cohomology ring 
\begin{equation*}\label{H*}H^*\cong\left\{
                            \begin{array}{ll}
                              \Q  & \ast =0, 4k; \\
                              \Q^{r}& \ast = 2k;\\
                               0  & \hbox{ otherwise}
                            \end{array}
                          \right.
\end{equation*}
with signature $\sigma$ and Euler characteristic $\chi$, such that for some fundamental class $\mu\in H_0\cong \Q$, the bilinear form $\lambda: H^{2k}\times
H^{2k}\rightarrow \Q$ defined as
$\langle\cdot\cup\cdot, \mu\rangle$ is isomorphic to $a\langle1\rangle\oplus b\langle-1\rangle$ for some non-negative integers $a, b$. Then there exist basis classes $\alpha_1, \ldots \alpha_a, \overline{\alpha}_1, \ldots \overline{\alpha}_b\in H^{n/2}$ such that  $\langle \alpha_i\alpha_j,\mu\rangle=\delta_{ij}$, $\langle \overline{\alpha}_i\overline{\alpha}_j,\mu\rangle=-\delta_{ij}$, and $\langle \alpha_i\overline{\alpha}_j,\mu\rangle=0$. Since $(\sigma, \chi)=(a-b, a+b+2)$, we also want to prescribe $a=\frac{\chi+\sigma-2}{2}\geq 0, b=\frac{\chi-\sigma-2}{2}\geq 0$, and at least one of $a,b$ is nonzero.

By Theorem \ref{acthm}, there exists a simply-connected almost complex manifold $M^{4k}$ such that $H^*(M;\Q)\cong (H^*, \sigma, \chi)$ if and only if there exist choice of class $c=1+c_{k}+c_{2k}\in H^*$ with $c_{k}\in H^{2k}\cong \Q^r$ and  $c_{2k}\in H^{4k}\cong \Q$, and a fundamental class $\mu\in H_0\cong \Q$ that satisfy conditions (i)--(v),

\begin{enumerate}[label={\upshape(\roman*)}]
\item

(intersection form) By our definition, $(H^*, \sigma, \chi)$ already satisfied this condition and $a=\frac{\chi+\sigma-2}{2}\geq 0, b=\frac{\chi-\sigma-2}{2}\geq 0$. Then any class $c_{k}$ and $c_{2k}$ can be expressed as
$$
\begin{cases}
c_{k}=\sum_{i=1}^a w_i\alpha_i+\sum_{j=1}^b \overline{w}_j \overline{\alpha_j},  \ \ \ \mbox{ for some }  w_i, \overline{w}_j\in\Q\\
c_{2k}=y\alpha_1^2 \ \ \mbox{ or } -y\overline{\alpha_1}^2 \ \ \ \mbox{ for some }  y\in\Q
\end{cases}.
$$ 
Then for all partitions $I$ of $2k$, the parings $\langle c_I, \mu\rangle$ are zero except possibly for
$$x=\langle c_{k}^2,\mu\rangle=\sum_{i=1}^a w_i^2-\sum_{j=1}^b \overline{w}_j^2\ \ \mbox{ and } \ \ y=\langle c_{2k},\mu\rangle \ \ \ \mbox{ for some }  w_i, \overline{w}_j, y\in\Q.$$

\item
(signature equation) 
\begin{enumerate}
\item
When $n=4k\equiv 0\pmod{8}$,  $p_i(c)=0$ for all $i$ except $p_{k/2}=(-1)^{\frac{k}{2}}2c_{k}$ and $p_{k}=c_{k}^2+2c_{2k}$. The $k$--th $\LP$ class is $\LP_{k}(p)=s_{\frac{k}{2},\frac{k}{2}}p_{\frac{k}{2}}^2+s_{k}p_{k}$, where the coefficients $s_{\frac{k}{2}}=\frac{2^{k}(2^{k-1}-1)|B_{k}|}{k!}$ and $s_{\frac{k}{2},\frac{k}{2}}=\frac{1}{2}(s_{\frac{k}{2}}^2-s_{k})$.
Then condition (ii) says
\begin{eqnarray}\label{signature8k}
\ \ \ \langle L_{k}(p), \mu\rangle=\langle 4s_{\frac{k}{2},\frac{k}{2}} c_{k}^2 +s_{k}(c_{k}^2+2c_{2k}),\mu\rangle
=(2s_{\frac{k}{2}}^2-s_{k})x+2s_{k}y=\sigma
\end{eqnarray}
\item
When $n=4k\equiv 4\pmod{8}$, $p_i(c)=0$ for all $i$ except $p_{k}=c_{k}^2-2c_{2k}$. The $k$--th $\LP$ class is $\LP_{k}(p)=s_{k} p_{k}$, where the coefficient $s_{k}=\frac{2^{2k}(2^{2k-1}-1)|B_{2k}|}{2k!}$. Then (ii) says
\begin{eqnarray}\label{signature8k+4}
\langle \LP_{k}(p), \mu\rangle
=s_{k}(x-2y)=\sigma
\end{eqnarray}
\end{enumerate}

\item [(iii)] and (iv)
(integrality relations from $U$ and $SU$ cobordism)
\begin{enumerate}
\item
When $n=4k\equiv 0\pmod{8}$, condition (iii) says $x=\langle c_{k}^2,\mu\rangle $ and $y=\langle c_{2k},\mu\rangle$ are integers that satisfy the integrality relations \eqref{BUrelation}, which have been simplified in Proposition \ref{U-relation-thm} to $\langle \td, \mu\rangle\in \Z$, $\langle e_1^c\mathord{\cdot} \td, \mu\rangle\in\Z$ and $\langle e_1^ce_1^c\mathord{\cdot} \td, \mu\rangle\in\delta_k\Z$.

\item
When  $n=4k\equiv 4\pmod{8}$, since $H^2=0$, we are in the case $c_1=0$, so condition (iii) and (iv) says $x=\langle c_{k}^2,\mu\rangle $ and $y=\langle c_{2k},\mu\rangle$ are integers that satisfy the integrality relations \eqref{BUrelation} and \eqref{BSUrelation4mod8}. Combining Proposition \ref{U-relation-thm} and \ref{SU-relation-thm}, the relations are satisfied if and only if $\langle \td, \mu\rangle\in 2\Z$, $\langle e_1^c\mathord{\cdot} \td, \mu\rangle\in\Z$ and $\langle e_1^ce_1^c\mathord{\cdot} \td, \mu\rangle\in\delta_k\Z$.

\end{enumerate}

\item [(v)]
(Euler characteristic) condition (v) requires $y=\langle c_{2k}, \mu\rangle=\chi$.  
 
\end{enumerate}

\begin{remark}
Note that because of condition (v), we can not presumably choose an orientation $\mu$ to prescribe $a\geq b$ or $b\geq a$. This is essentially because any complex vector bundle gives the underlying real vector bundle a canonical preferred orientation. One can recall the fact that $\CP^2$ admits an almost complex structure but $\overline{\CP^2}$ does not.\\
\end{remark}

\subsection{Realization theorem} \label{subsec:4.1}\hfill\\

We phrase the realization theorem in dimension $n\equiv 0 \pmod{8}$ and $n\equiv 4\pmod{8}$ respectively. A rational cohomology ring $H^*$ under our Betti number assumption is realizable by an almost complex manifold if and only if the signature $\sigma$ and the Euler characteristic $\chi$ satisfy a set of integrality conditions.
 
 \begin{theorem}\label{existence8k}
There exists an $n=8k$--dimensional simply-connected closed almost complex manifold $M$ with Betti number $b_i=0$ except $b_0=b_n=1$ and $b_{n/2}\geq 1$, signature $\sigma$, and Euler characteristic $\chi$ if and only if $a=\frac{\chi+\sigma-2}{2}\geq 0, b=\frac{\chi-\sigma-2}{2}\geq 0$, and the following integrality relations are satisfied for some integer $x$.
\begin{subnumcases}
\ \langle L_{2k}, \mu\rangle=(2s_{k}^2-s_{2k})x+2s_{2k}\chi=\sigma \ \ {\scriptstyle \textup{with}  \ \    s_{k}=\frac{2^{2k}(2^{2k-1}-1)|B_{2k}|}{(2k)!}}    \label{thm-L-8k}\\
\langle\td, \mu\rangle=\frac{1}{2}(t_{2k}^2-t_{4k})x+t_{4k}\chi\in \Z \ \ {\scriptstyle \textup{with}\ \   t_k=\frac{B_k}{k!}\label{thm-Td-8k}}\\
\langle e_1^c\mathord{\cdot} \td, \mu\rangle=\left[\df{-t_{2k}}{(2k-1)!}+\df{1}{2(4k-1)!}\right]x -\df{\chi}{(4k-1)!}\in\Z \label{thm-e1Td-8k}\\
\langle e_1^ce_1^c\mathord{\cdot} \td, \mu\rangle=\df{x}{[(2k-1)!]^2}\in\delta_k\Z  \ \ {\small \textup{with}  \ \delta_k=\left\{
                            \begin{array}{ll}
                              1  & k=1\\
                              2& k>1
                            \end{array}
                          \right.}
\label{thm-e1e1Td-8k}
\end{subnumcases}

\end{theorem}

 \begin{theorem}\label{existence8k+4}
There exists an $n=8k+4 (k>0)$--dimensional simply-connected closed almost complex manifold $M$ with Betti number $b_i=0$ except $b_0=b_n=1$ and $b_{n/2}\geq 1$, signature $\sigma$, and Euler characteristic $\chi$ if and only if $a=\frac{\chi+\sigma-2}{2}\geq 0, b=\frac{\chi-\sigma-2}{2}\geq 0$, and the following integrality relations are satisfied for some integer $x$.
\begin{subnumcases}
\ \langle L_{2k+1}, \mu\rangle=s_{2k+1}(x-2\chi)=\sigma \label{thm-L-8k+4}\\ 
\langle\td, \mu\rangle=-\frac{1}{2}t_{4k+2}(x-2\chi)\in2\Z  \label{thm-Td-8k+4}\\ 
\langle e_1^c\mathord{\cdot} \td, \mu\rangle=\df{1}{2(4k+1)!}(x-2\chi)\in\Z \\ \label{thm-e1Td-8k+4}
\langle e_1^ce_1^c\mathord{\cdot} \td, \mu\rangle=\df{x}{[(2k)!]^2}\in2\Z \label{thm-e1e1Td-8k+4}
\end{subnumcases}
\end{theorem}

\begin{proof}[Proof of Theorem~\ref{existence8k} and \ref{existence8k+4}]
We have explained in the beginning of this section that conditions (ii)-(v) in Theorem \ref{acthm} applied in our prescribed rational cohomology ring result in \eqref{thm-L-8k}--\eqref{thm-e1e1Td-8k} in dimension $n=8k$ and \eqref{thm-L-8k+4}--\eqref{thm-e1e1Td-8k+4} in dimension $n=8k+4$, where the integer $x$ equals the Chern number $\langle c_{2k}^2(\tau_M),[M]\rangle$ of the realizing manifold $M$. We are only left to show that any integer $x$ satisfying these conditions can be written as $\sum_{i=1}^a w_i^2-\sum_{j=1}^b \overline{w}_j^2$ for some rational numbers $w_i, \overline{w}_j$, this is true as long as $a\neq 0$ and $b\neq 0$ because any integer $x$ can be written as a difference of two rational squares. In fact, if $\sigma$ and $\chi$ are integers that support any solution $x$ satisfying the integrality conditions, the integers $a=\frac{\chi+\sigma-2}{2}$ and $b=\frac{\chi-\sigma-2}{2}$ are automatically nonzero. For dimension $n=8$, this is shown in the proof of Proposition \ref{thm-8dim}. For any dimension $n=8k$ with $k>1$ or $n=8k+4$ with $k>0$, this can be seen from Theorem \ref{sigma-euler-div-8k} and Theorem \ref{sigma-euler-div-8k+4} respectively, which says the integrality conditions requires $\sigma$ and $\chi$ to be divisible by 4, therefore $a$ and $b$ must be nonzero. 

\end{proof}

\subsection{Examples in dimension 8, 12, and 16}\label{subsec:4.2} \hfill\\

For dimension $n=4k$ with $k=1,2,3$, we calculate the congruence relations among $\sigma$ and $\chi$, and work out some examples.

\begin{proposition}\label{thm-8dim}
There exists an $8$--dimensional simply-connected almost complex manifold with $b_i=0$ except $b_0=b_8=1$ and $b_{4}\geq 1$, signature $\sigma$, and Euler characteristic $\chi$ if and only if $\sigma\equiv 0\bmod{2}$, $\chi\equiv 0\bmod{6}$,  $3\sigma-\chi\equiv 0\bmod{48}$, and $a=\frac{\chi+\sigma-2}{2}> 0, b=\frac{\chi-\sigma-2}{2}> 0$.
\end{proposition}

\begin{proof}
For dimension $n=8$, conditions \eqref{thm-L-8k}--\eqref{thm-e1e1Td-8k} in Theorem \ref{existence8k} says:
\begin{equation*}
\begin{cases}
\langle L_{2}, \mu\rangle=\frac{1}{45}(3x+14\chi)=\sigma \\
\langle\td, \mu\rangle=\frac{1}{720}(3x-\chi)\in\Z\\
\langle e_1^c\mathord{\cdot} \td, \mu\rangle=-\frac{1}{6}\chi\in\Z \\
\langle e_1^ce_1^c\mathord{\cdot} \td, \mu\rangle=x\in\Z
\end{cases}
\Longleftrightarrow
\begin{cases}
3x+14\chi=45\sigma \\
3x-\chi=720m  \text{ for } m\in\Z\\
\chi=6s  \text{ for } s\in\Z^+
\end{cases}.
\end{equation*}
The system has integer solution $x$ if and only if
$\chi=6s$, $\sigma=2s+16m$ for $s\in \Z^+, m\in\Z$, and  $x=2s+240m$. The rational intersection form $a\langle1\rangle\oplus b\langle-1\rangle$ requires $a=\frac{\chi+\sigma-2}{2}=4s+8m-1\geq 0$ and $b=\frac{\chi-\sigma-2}{2}=2s-8m-1\geq 0$, which is equivalent to $\frac{-4s+1}{8}\leq m\leq\frac{2s-1}{8}$ and $s$ can take any positive integer value. Note that $a, b\neq 0$ for any $s\in\Z^+$ and $m\in\Z$.\end{proof}
Below is a list of realizable rational cohomology ring $(H^*,\sigma, \chi)$ with the smallest few possible values of Euler characteristic. Note that $\chi=48$ would be the smallest value to realize signature $\sigma=0$, with $(a,b)=(23, 23)$.
\begin{small}
\begin{eqnarray*}
&&\chi=6,\sigma=2;\ \ \ (a,b)=(3, 1)^*; \\
&&\chi=12,\sigma=4;\ \ \ (a,b)=(7, 3)^*; \\
&&\chi=18,\sigma=-10, 6;\ \ \ (a,b)=(3, 13), (11,5)^*; \\
&&\chi=24,\sigma= -8, 8;\ \ \ (a,b)=(7, 15), (15, 7)^*; \\
&&\chi=30,\sigma=-22, -6, 10, 26;\ \ \ (a,b)=(3, 25), (11, 17), (19, 9)^*, (27, 1); 
\end{eqnarray*} 
\end{small}
\begin{remark}
Note that in dimension 8, the prescribed rational cohomology $(H^*, \sigma, \chi)$ agree with that of $\#_a\HP^2\#_b\overline{\HP^2}$. 
One can compare the realizable examples listed above with the discussion in \cite[6.1]{Mil21}, which calculated that the connected sums of quaternionic projective planes $\#_a\HP^2\#_b\overline{\HP^2}$ admits almost complex structure if and only if $(a,b)=(4n+3,2n+1)$ for some $n\in\Z$, that is to say the integral (no--torsion) version of the rational cohomology ring $(H^*, \sigma, \chi)$ is realizable if and only if $\sigma\equiv 0\bmod{2}$, $\chi\equiv 0\bmod{6}$, and $3\sigma-\chi=0$. These examples are marked with $*$ in the above list of $(a,b)$.
\end{remark}

\begin{proposition}\label{thm-16dim}
There exists a $16$--dimensional simply-connected almost complex manifold with $b_i=0$ except $b_0=b_{16}=1$ and $b_{8}\geq 1$, signature $\sigma$, and Euler characteristic $\chi$ if and only if $\sigma\equiv 0\bmod{8}$, $\chi\equiv 0\bmod{120}$, $15\sigma-\chi\equiv 0\bmod{1952}$, and $a=\frac{\chi+\sigma-2}{2}> 0, b=\frac{\chi-\sigma-2}{2}> 0$.
\end{proposition}

\begin{proof}
For dimension $n=16$, \eqref{thm-L-8k}--\eqref{thm-e1e1Td-8k} in Theorem \ref{existence8k} says:
\begin{equation*}
\begin{cases}
\langle L_{2}, \mu\rangle=\frac{1}{14175}(305x+762\chi)=\sigma \\
\langle\td, \mu\rangle=\frac{1}{3628800}(5x-3\chi)\in\Z\\
\langle e_1^c\mathord{\cdot} \td, \mu\rangle=\frac{1}{15120}(5x-3\chi)\in\Z \\
\langle e_1^ce_1^c\mathord{\cdot} \td, \mu\rangle=\frac{x}{36}\in 2\Z\nonumber
\end{cases}
\Longleftrightarrow
\begin{cases}
305x+762\chi=3^45^27\sigma \\
5x-3\chi=2^83^45^27m \  \text{ for } m\in\Z\\
x=2^33^2\ell \  \text{ for } \ell\in\Z
\end{cases}
\end{equation*}
The system has integer solution for $x$ if and only if
$\chi=120s$, $\sigma=8(s+1952m)$ for $s\in \Z^+, m\in\Z$, and $x=72s+725760m$. 
The rational intersection form requires $a=\frac{\chi+\sigma-2}{2}=64s+7808m-1\geq 0$ and $b=\frac{\chi-\sigma-2}{2}=56s-7808m-1\geq 0$, which is equivalent to $\frac{-64s+1}{7808}\leq m\leq \frac{56s-1}{7808}$ and $s$ can take any positive integer value. Note that $\chi$ and $\sigma$ are both divisible by 4,  so $a, b\neq 0$ for any value of $s$ and $m$.

If $H^*$ is realizable, the smallest two possible values of $\chi$ are $\chi=120$, with $\sigma=8, (a,b)=(64, 56)$, and $\chi=240$, with $\sigma=16, (a,b)=(127, 111)$.  
\end{proof}

\begin{proposition}\label{thm-12dim}
There exists a $12$--dimensional simply-connected almost complex manifold with $b_i=0$ except $b_0=b_{12}=1$ and $b_{6}\geq 1$, signature $\sigma$, and Euler characteristic $\chi$ if and only if the signature $\sigma\equiv 0\bmod{(2^8)(31)}$, $\chi\equiv 0\bmod{4}$, and $a=\frac{\chi+\sigma-2}{2}> 0, b=\frac{\chi-\sigma-2}{2}> 0$.
\end{proposition}

\begin{proof}
For dimension $n=12$, conditions \eqref{thm-L-8k+4}--\eqref{thm-e1e1Td-8k+4}  in Theorem \ref{existence8k+4} says:
\begin{equation*}
\begin{cases}
\langle L_{3}, \mu\rangle=\frac{62}{945}(x-2\chi)=\sigma \\
\langle\td, \mu\rangle=\frac{-1}{60480}(x-2\chi)\in2\Z\\
\langle e_1^c\mathord{\cdot} \td, \mu\rangle=\frac{1}{240}(x-2\chi)\in\Z \\
\langle e_1^ce_1^c\mathord{\cdot} \td, \mu\rangle=\frac{x}{4}\in2\Z
\end{cases}
\Longleftrightarrow
\begin{cases}
62(x-2\chi)=3^35^17\sigma \\
x-2\chi=2^73^35^17m  \text{ for } m\in\Z\\
x=8\ell   \text{ for } \ell\in\Z
\end{cases}.
\end{equation*}
The system has solution if and only if
$\chi=4s$, $\sigma=(2^8)(31)m$ for $s\in \Z^+, m\in\Z$, and $x=8s+(2^7)(3^3)(5)(7)m$. The rational intersection form $a\langle1\rangle\oplus b\langle-1\rangle$ requires $a=\frac{\chi+\sigma-2}{2}=2s+2968m-1\geq 0$ and $b=\frac{\chi-\sigma-2}{2}=2s-2968m-1\geq 0$, which is equivalent to $s\geq 1984|m|+1$ and $m$ can take any integer value. Again note that $\chi$ and $\sigma$ are both divisible by 4, so $a, b\neq 0$ for any value of $s$ and $m$.

If $\sigma=0$, $H^*$ is realizable if and only if $\chi=4s$ for $s\geq 1$. The smallest nonzero absolute value of the signature is $|\sigma|=7936$, in this case, $H^*$ is realizable if and only if $\chi=4s$ for $s\geq 1985$.
\end{proof}

\subsection{Divisibility of the signature and Euler characteristic} \label{subsec:4.3}\hfill\\

We obtain the following divisibility results on $\sigma$ and $\chi$ in dimension $n=8k$ and $n=8k+4$ respectively. The approach is similar to that applied in \cite{FS16} and \cite{KS17}, by examining the coefficients in the signature equation, Todd genus, and the other integrality conditions of Chern numbers. Note that these divisibility conditions are necessary for any, not only simply-connected, almost complex manifold realizing the prescribed rational cohomology ring.

\begin{theorem}\label{sigma-euler-div-8k}

If $M^n$ is an $n=8k$--dimensional almost complex manifold with Betti number $b_i=0$ except $b_0=b_n=1$ and $b_{n/2}\geq 1$, its signature $\sigma$ and Euler characteristic $\chi$ must satisfy the following divisibility conditions, where $\nu_2(k)$ denotes the 2-adic order of $k$, $\textup{wt}(k)$ denotes the number of ones in the binary representation of $k$.
\begin{eqnarray} 
&&\nu_2(\sigma) \geq 4k - 2\nu_2(k) - 3, \label{sigma-div-8k}\\
&&\nu_2(\chi) \geq 4k - 2\nu_2(k) - 2\textup{wt}(k)-2. \label{euler-div-8k}
\end{eqnarray}For example, in dimension $n=8$, $\sigma\equiv 0\bmod{2}$,

\hspace{1.75cm} in dimension $n=16$, $\sigma\equiv 0\bmod{2^3}$ and $\chi\equiv 0\bmod{2^2}$,

\hspace{1.75cm} in dimension $n=24$, $\sigma\equiv 0\bmod{2^9}$ and $\chi\equiv 0\bmod{2^6}$.

\end{theorem}

\begin{proof}
It has been shown in Proposition \ref{thm-8dim} that any $n=8$-dimensional almost complex manifold realizing $H^*$ must have the signature $\sigma\equiv 0\bmod{2}$, so \ref{sigma-div-8k} holds true for $k=1$.

If $(H^*, \sigma, \chi)$ is realizable by an $n=8k$--dimensional almost complex manifold, conditions \eqref{thm-L-8k}, \eqref{thm-Td-8k} and \eqref{thm-e1e1Td-8k} in Theorem \ref{existence8k}  says there must exist an integer $x$ such that 
\begin{subnumcases}
\ 2s_{k}^2x-s_{2k}(x-2\chi)=\sigma \label{signature-div-signature}\\
t_{2k}^2x-t_{4k}(x-2\chi)=2m \ \   \text{ for } m\in\Z  \label{signature-div-Td}\\
x=[(2k-1)!]^2\,2\ell \label{signature-div-e1e1Td}   \text{ for } \ell\in\Z
\end{subnumcases}
Since $s_{2k} = -2^{4k} (2^{4k-1} - 1) t_{4k}$,  we multiply \ref{signature-div-Td} by $2^{4k} (2^{4k-1} - 1)$ and add it to \ref{signature-div-signature} to eliminate the $(x-2\chi)$ term, after simplifying and plugging in \ref{signature-div-e1e1Td}, we have
\[\underbrace{2^{4k}(2^{4k}-2^{2k+1} + 1)\left(\df{|B_{2k}|}{2k}\right)^2 2}_{L1}\ell= \underbrace{2^{4k+1} (2^{4k-1} - 1)}_{R1} m + \sigma.\]
By the 2-adic order formula $\nu_2\left(\frac{|B_{2k}|}{2k}\right)=-(\nu_2(k)+2)$ (\cite[section 3]{KS17}), we have 
$$\nu_2(L1)=4k-2\nu_2(k)-3, \ \ \nu_2(R1)=4k+1>\nu_2(L1).$$
If $\sigma\neq 0$, then $\nu_2(\sigma) \geq\textup{min}\{\nu_2(L1), \nu_2(R1)\}=\nu_2(L1)$, which proved \eqref{sigma-div-8k}.
To show the Euler characteristic divisibility in \eqref{euler-div-8k}, we use conditions \eqref{thm-e1Td-8k} and \eqref{thm-e1e1Td-8k}, 
\begin{subnumcases}
\ \left[\df{-t_{2k}}{(2k-1)!}+\df{1}{2(4k-1)!}\right]x -\df{\chi}{(4k-1)!}=m \ \  \text{ for } m\in\Z \label{euler-div-e1Td}\\
x=[(2k-1)!]^2\,2\ell \ \  \text{ for } \ell\in\Z \label{euler-div-e1e1Td}
\end{subnumcases}
Multiply \eqref{euler-div-e1Td} by $(4k-1)!$ and plugin \eqref{euler-div-e1e1Td}, we get 
\[\underbrace{-\df{|B_{2k}|}{2k}(4k-1)!\, 2}_{L1}\ell+\underbrace{[(2k-1)!]^2}_{L2}\ell-\chi= \underbrace{(4k-1)!}_{R1} m.\]
By the 2-adic order formula $\nu_2\!\left(\!\frac{|B_{2k}|}{2k}\!\right)=-(\nu_2(k)+2)$, 
$\nu_2[(2k-1)!]=2k-\nu_2 (k)-\textup{wt}(k)-1$, and 
$\nu_2[(4k-1)!]=4k-\nu_2(k)-\textup{wt}(k)-2$ 
(computed in \cite[section 3]{KS17}), we have 
\begin{eqnarray*}
&&\nu_2(L1)=4k-2\nu_2(k)-\textup{wt}(k)-3, \ \ \nu_2(L2)=4k-2\nu_2(k)-2\textup{wt}(k)-2,\\
&& \ \nu_2(R1)=4k-\nu_2(k)-\textup{wt}(k)-2.
 \end{eqnarray*}
Then $\nu_2(\chi)\geq\textup{min}\{\nu_2(L1), \nu_2(L2), \nu_2(R1)\}=\nu_2(L2)$. 
\end{proof}

\begin{theorem}\label{sigma-euler-div-8k+4}

If $M^n$ is an $n=8k+4\  (k>0)$--dimensional almost complex manifold with Betti number $b_i=0$ except $b_0=b_n=1$ and $b_{n/2}\geq 1$, its  signature $\sigma$ and Euler characteristic $\chi$ must satisfy the following divisibility conditions:
\begin{eqnarray}
&&\sigma\equiv 0 \bmod{2^{4k+4}(2^{4k+1}-1)} \label{sigma-div-8k+4}\\
&&\chi\equiv 0 \bmod{[(2k)!]^2} \label{euler-div-8k+4}
\end{eqnarray}
In particular, the 2--adic order of $\sigma$ and $\chi$ satisfy:
\begin{eqnarray}
&&\nu_2(\sigma)\geq 4k+4 \label{sigma-div-8k+4.}\\
&&\nu_2(\chi)\geq 4k-2\textup{wt}(k) \label{euler-div-8k+4.}
\end{eqnarray}

\end{theorem}

\begin{proof}
The Todd genus condition \eqref{thm-Td-8k+4} says  $-t_{4k+2}(x-2\chi)=4m$ for some $m\in\Z$. Since $s_{2k+1} = 2^{4k+2} (2^{4k+1} - 1) t_{4k+2}$, multiplying this equation by $2^{4k+2} (2^{4k+1} - 1)$ and adding it to the signature equation \eqref{thm-L-8k+4} yields
$0=\sigma + 2^{4k+2} (2^{4k+1} - 1) 4m$, which proved the claim on $\sigma$.

By \eqref{thm-e1e1Td-8k+4}, $x=2[(2k)!]^2\ell$ for some $\ell\in\Z$, plug this into the Todd genus condition \eqref{thm-Td-8k+4} which says $x-2\chi=2(4k+1)!m$ for some $m\in\Z$, we have $[(2k)!]^2\ell-\chi=(4k+1)!m$, since $\frac{(4k+1)!}{[(2k)!]^2}=(4k+1)\binom{4k}{2k}\in\Z$, \eqref{euler-div-8k+4} is proved. $\nu_2((2k)!)=2k-\textup{wt}(2k)=2k-\textup{wt}(k)$ gives \eqref{euler-div-8k+4.}.

\end{proof}

Theorem \ref{sigma-euler-div-8k} and Theorem \ref{sigma-euler-div-8k+4} imply the following.

\begin{theorem}\label{nonexist-1}
An $n=4k (k>1)$--dimensional closed almost complex manifold with Betti number $b_i=0$ except $b_0=b_n=1$,$b_{n/2}\geq 1$ must have even signature $\sigma$ and even Euler characteristic $\chi$, i.e., the middle Betti number $b_{n/2}$ must be even.
\end{theorem} 

\begin{corollary}\label{nonexist-2}
In dimension greater than 4, the rarely existing rational projective plane \textup{(}smooth manifold whose Betti number $b_0=b_{n/2}=b_n=1$\textup{)} does not admit any almost complex structure.  Equivalently, there does not exist any closed almost complex manifold whose sum of Betti numbers equals three. 
\end{corollary}

\begin{remark}\label{rmk-nonexist-2}  
One can refer to \cite{Su09}, \cite{Su14}, \cite{FS16}, \cite{KS17} for results on existence of rational projective planes. In summer 2018, upon seeing the arXiv paper \cite{AM19} showing that an almost complex manifold with sum of Betti numbers three can only exists in dimension $n=2^a$, the author had worked out the nonexistence result in Corollary \ref{nonexist-2} using the signature equation and the integrality condition that $\langle e_1^ce_1^c\mathord{\cdot} \td, \mu\rangle=\frac{c_{2k}^2[M^{8k}]}{[(2k-1)!]^2}\in\Z$. In 2019, the author was informed that Jiahao Hu had also obtained this result independently using the signature equation and the integrality of Todd genus \cite{JH21}. Now we know more generally, there does not exist almost complex manifold (above dimension 4) with Betti numbers concentrated in the middle dimension and the sum of Betti numbers being odd. The lower bound of the 2-adic order of the Euler characteristic increases with respect to the dimension of the manifold.
\end{remark}



\noindent

\begin{thebibliography}{2}


\bibitem[Las63]{Lashof63}
R. Lashof.
\newblock {\em Poincar\'{e} Duality and Cobordism}
\newblock  Transactions of the American Mathematical Society, Vol. 109, No. 2, (Nov., 1963), pp. 257-277





\bibitem[Sto65I]{Stong65I}
R. E. Stong.  
\newblock {\em Relations among characteristic numbers I.}
\newblock  Topology, 4 1965 267--281.


\bibitem[Sto65II]{Stong65II}
R. E. Stong.  
{\em Relations among characteristic numbers II.}
\newblock  Topology, 5 1965 133-148.



\bibitem[Sut65]{Sutherland65}
Sutherland, W. A.
{\em A note on almost complex and weakly complex structures.}
\newblock  J. London Math. Soc. 40 (1965), 705–712.  

\bibitem[Tho67]{Thomas67}
Thomas, Emery
{\em Complex structures on real vector bundles.}
\newblock Amer. J. Math. 89 (1967), 887–908. 



\bibitem[Sto68]{Stong68}
R. E. Stong.  
\newblock {\em Notes on cobordism theory.} 
\newblock Mathematical notes Princeton University
Press, 1968.


\bibitem[Quil69]{Quillen69}
D. Quillen.  
\newblock {\em Rational homotopy theory.}
\newblock  Annals of Mathematics, 90 (2) 1969 205-295.


%
%
%

\bibitem[MH73]{MilnorHusemoller73}
J. Milnor; D. Husemoller.
\newblock {\em Symmetric Bilinear Form.}
\newblock Springer-Verlag, 1973.


\bibitem[MS74]{MilnorStasheff}
J.W. Milnor and J.D. Stasheff.
\newblock {\em Characteristic classes}.
\newblock Annals of Mathematics Studies, No. 76. Princeton University Press, 1974. 
%
%




\bibitem[Quin74]{Quinn74}
F. ~Quinn,
\newblock {\em Semifree group actions and surgery on PL homology manifolds.}
\newblock Geometric topology (Proc. Conf., Park City, Utah, 1974),pp. 395--414. Lecture Notes in Math., Vol. 438, Springer, Berlin, 1975.




\bibitem[Bar76]{Barge76}
J.~Barge.
\newblock {\em Structures diff\'erentiables sur les types d'homotopie rationnelle simplement connexes.}
\newblock  Ann. Sci. \'Ecole Norm. Sup. (4), 9(4):469--501, 1976.






\bibitem[Sul77]{Sullivan77}
D.~Sullivan.
\newblock {\em Infinitesimal computations in topology.}
\newblock  Inst. Hautes \'Etudes Sci. Publ. Math. (47):269--331 (1978), 1977.

\bibitem[And77]{Anderson77}
G. ~Anderson.
\newblock {\em Surgery with coefficients.} 
\newblock Lecture Notes in Mathematics, Vol. 591. Springer-Verlag,
Berlin-New York, 1977.




\bibitem[TW78]{TaylorWilliams}
L. ~Taylor; B, ~Williams.
\newblock {\em Local surgery: foundations and applications.}
\newblock  Algebraic topology, Aarhus 1978, Lecture Notes in Math., Vol. 763, pp. 673--695, Springer, 1979.


%
%
%
%
%
%
%
%
%
%
%

\bibitem[HBJ92]{HBJ92}
F. ~Hirzebruch; T. ~Berger; R. ~Jung.
\newblock {\em Manifolds and Modular Forms}
\newblock  Aspects of Mathematics, E20.,
1992.


\bibitem[Sla95]{Sla95}
I.S. ~Slavutskii.
\newblock {\em A note on Bernoulli numbers},
\newblock J. Number Theory 53 (1995), no. 2, 309–310.

\bibitem[GM00]{GM00}
S. M\"{u}ller, S and H. Geiges, 2000. 
\newblock {\em Almost complex structures on 8-manifolds}, \newblock L’Enseignement Math\'{e}matique, (2) 46 (2000), no. 1-2, 95–107.


\bibitem[PL07]{PL07}
S. Papadima and  L. Paunescu.
\newblock{\em Closed manifolds coming from Artinian complete intersections}, 
\newblock Trans. Amer. Math. Soc. 359 (2007), no. 6


\bibitem[Su09]{Su09}
Z.~Su.
\newblock {\em Rational homotopy type of manifolds}, 
\newblock  Ph.D. dissertation, Indiana University, 2009. Thesis advisor: James F. Davis


\bibitem[Su14]{Su14}
Z.~Su.
\newblock {\em Rational analogs of projective planes}.
\newblock  Algebraic $\&$ Geometric Topology,  14 (2014) 421-438


\bibitem[FS16]{FS16}
J.~Fowler; Z.~Su.
\newblock {\em Smooth manifolds with prescribed rational cohomology ring.}
\newblock  Geometriae Dedicata, 182(1)  (2016), 215-232

\bibitem[KS17]{KS17}
L. Kennard; Z. Su.
\newblock {\em On dimensions supporting a rational projective plane.}
\newblock  Journal of Topology and Analysis, (2017) 1-21


\bibitem[AM19]{AM19}
M. Albanese and A. Milivojevi\'{c}.
\newblock{ \em On the minimal sum of Betti numbers of an almost complex manifold.}
\newblock Differential Geometry and its Applications, 62, pp.101-108.

\bibitem[JH21]{JH21}
J. Hu.
\newblock{ \em Almost complex manifolds with total Betti number three.}

\newblock  https://doi.org/10.48550/arXiv.2108.06067

\bibitem[Mil21]{Mil21}
A. Milivojevi\'{c}.
\newblock{\em On the characterization of rational homotopy types and Chern classes of closed almost complex manifolds.}
\newblock Doctoral dissertation, State University of New York at Stony Brook. To appear in journal.

\end{thebibliography}
\end{document}